\documentclass[reqno]{amsart}
\usepackage[english]{babel}
\usepackage{amscd,amssymb,amsmath,amsfonts,latexsym,amsthm}
\usepackage{ulem}
\usepackage{inputenc}
\usepackage{graphicx,color}
\newtheorem{thm}{Theorem}[section]
\newtheorem{cor}[thm]{Corollary}
\newtheorem{lem}[thm]{Lemma}
\newtheorem{prop}[thm]{Proposition}

\numberwithin{equation}{section}

\begin{document}

\title[Infinitesimal deformations and central extensions of the $n$-th Schr\"{o}dinger algebra]{Infinitesimal deformations and central extensions of the $n$-th Schr\"{o}dinger algebra}

\author[D.E.Jumaniyozov]{Doston Jumaniyozov$^{1,2}$}

\author[S.A. Sheraliyeva]{Surayyo Sheraliyeva$^1$}

\address{$^1$V.I.Romanovskiy Institute of Mathematics Uzbekistan Academy of Sciences.}

\address{$^2$National University of Uzbekistan named after Mirzo Ulugbek.}

\

\email{{\tt dostondzhuma@gmail.com, abdiqodirovna@mail.ru}}


\begin{abstract} In this paper, we describe infinitesimal deformations of the $n$-th Schr\"{o}dinger algebra $\mathfrak{sch}_n$ and the factor algebra $\mathfrak{g}_n$. We prove that the second cohomology group with values in the adjoint representation of the $n$-th Schr\"{o}dinger algebra and the factor algebra of the $n$-th Schr\"{o}dinger algebra by its one-dimensional center vanishes for $n\geq3.$ On the other hand, it is shown that for $n = 2$ in both cases
this cohomology space is one-dimensional.
Moreover, we investigate the central extensions of these algebras and compute their second cohomology groups with values in the trivial representation.

\end{abstract}



\maketitle


\section{Introduction}

In the structural theory of Lie algebras, cohomology and deformation theory play a crucial role.
The classical deformation theory of associative and Lie algebras began
with the works of Gerstenhaber \cite{Gest} and
Nijenhuis-Richardson \cite{NR} in the 1960s. They studied
one-parameter deformations and established the connection between
the cohomology and infinitesimal deformations of Lie algebras.
The foundation of Lie algebra cohomology theory traces back to Cartan's work.
However, its development as an independent research area is attributed to Chevalley and Eilenberg \cite{ChEg}, Koszul \cite{Ksz}, and Hochschild and Serre \cite{HchSe}.
Cartan and Eilenberg \cite{CEil} provided a unified framework for the cohomology theories of groups, associative algebras, and Lie algebras.

Since the geometric description of finite-dimensional Lie algebras is an important problem, geometric approaches to cohomology have been developed in recent years.
It is a well-known result in algebraic geometry that any algebraic variety (evidently, algebras defined via identities form an algebraic variety) is a union of a finite number of irreducible components. Rigid algebras are those whose orbits under the action of an appropriate general linear group are open sets in the Zariski topology.
It is known that a Lie algebra with vanishing second cohomology group with values in the adjoint representation is rigid and the closures of the orbits of rigid algebras form the irreducible components of the variety (\cite{NR}, Theorem 7.1).
This has been a motivation for many works which focused on discovering algebras with open orbits and describing important properties of such algebras \cite{Ru, Makhlouf, Gest, NR}.

It is a well-known fact that any finite-dimensional Lie algebra over a field of characteristic zero is a semidirect product of a solvable ideal and a semisimple Lie subalgebra (Levi's theorem). Similar to semisimple Lie algebras, there are many applications of non-semisimple Lie algebras and groups in physics, e.g., the Poincar\'{e} algebra and group. Other important Lie algebras are the well-known Schr\"{o}dinger algebra in $(n+1)$-dimensional space-time and the $n$-th Schr\"{o}dinger algebra. These algebras are defined as the semidirect product of a semisimple Lie algebra and a Heisenberg algebra. Recently, these Schr\"{o}dinger algebras have been an object of many research works.

For example, all derivations of the Schr\"{o}dinger algebra in $(n+1)$-dimensional space-time have been determined in \cite{WuT}, and it has been shown that this algebra admits an outer derivation.
The derivations and automorphisms of the $n$-th Schr\"odinger algebra were described in \cite{LY}.
The first and second cohomology groups of the Schr\"{o}dinger algebra in $(1+1)$-dimensional space-time
with values in the trivial module and the finite-dimensional irreducible modules are computed in \cite{WuZh}.
The second cohomology group with values in the adjoint module of the Schr\"odinger algebra in $(n+1)$-dimensional space-time is determined in \cite{CEil}.
In this paper, we compute the second cohomology groups for the $n$-th Schr\"{o}dinger algebra $\mathfrak{sch}_n$  with values in the trivial and adjoint module, respectively.

The organization of the paper is as follows. Section 2 provides some necessary background information on Lie algebra cohomology, the Hochschild-Serre factorization theorem (Theorem 13 in \cite{HchSe}),
and the definition of the $n$-th Schr\"{o}dinger algebra. Section 3 contains the computations of the second cohomology groups of the $n$-th Schr\"{o}dinger algebra with values  in the trivial and adjoint modules. Namely, in this section, we prove that $H^2(\mathfrak{sch}_n,\mathfrak{sch}_n)=0$ for $n\geq3$ and $\dim(H^2(\mathfrak{sch}_2,\mathfrak{sch}_2))=1$.
In Section 4, we investigate the factor algebra of the $n$-th Schr\"{o}dinger algebra by its one-dimensional center.
In particular, we compute the second cohomology group of this algebra, considering both the trivial and adjoint modules, for the case where $n\geq2.$

Unless otherwise stated, any Lie algebra considered in this work is finite-dimensional
and defined over $ \mathbb{C}.$ Throughout the paper, the expression $\langle x_1,\dots,x_n\rangle$ means the linear subspace generated by the elements $x_1,\dots,x_n\in\mathfrak{g}.$

\section{Preliminaries}

A {\it Lie algebra} is a vector space $\mathfrak{g}$ endowed with bilinear mapping $[-,-]:\mathfrak{g}\times\mathfrak{g}\to\mathfrak{g}$ satisfying the following identities:
\[[x,x]=0,\]
\[[[x,y],z]+[[y,z],x]+[[z,x],y]=0.\]

The center of a Lie algebra $\mathfrak{g}$ is defined as
$$Z(\mathfrak{g})=\{x\in\mathfrak{g} \ | \ [x,y]=0,  \mbox{ for all } y\in\mathfrak{g}\}.$$

A linear operator $d$ on a Lie algebra $\mathfrak{g}$ is called a {\it derivation} if it satisfies the Leibniz rule, i.e., the identity
$$d([x,y])=[d(x),y]+[x,d(y)]$$
holds for any $x, y\in \mathfrak{g}.$ The set of all derivations of $\mathfrak{g}$ is denoted by ${\rm Der}(\mathfrak{g}).$
The linear operator ${\rm ad}_{x}$ defined by
${\rm ad}_{x}(y)=[x,y]$
is a derivation of $\mathfrak{g}$. Such derivations are called {\it inner derivations} and the set of all inner derivations is denoted by ${\rm Inn}(\mathfrak{g}).$

Let $\mathfrak{g}$ be a Lie algebra. A $\mathfrak{g}$-module is a vector space $M$ together with a homomorphism of $\mathfrak{g}$ into the Lie algebra of all linear operators on $M.$ We denote by $v.m$ the image of the element $m\in M$ under the linear operator which corresponds to the element $v\in\mathfrak{g}.$ The subspace of $M$ which consists of all $m\in M$ with $v.m=0$ for all $v\in\mathfrak{g}$ is denoted by $M^{\mathfrak{g}}.$
Set
$$C^0(\mathfrak{g},M)=M, \quad C^n(\mathfrak{g},M)={\rm Hom}(\wedge^{n}\mathfrak{g},M), \ n>0.$$

Namely, $C^n(\mathfrak{g},M)$ is the space of alternating $n$-multilinear maps. The elements of $C^n(\mathfrak{g},M)$ are called $n$-{\it cochains}. Now, define the differential $d^n:C^n(\mathfrak{g},M)\to C^{n+1}(\mathfrak{g},M).$ Let $\varphi\in C^n(\mathfrak{g},M)$ be an $n$-cochain. The image of the $n$-cochain $\varphi$ under the differential $d^n$ is defined as follows:
\begin{equation} \label{eq1} \begin{array}{ll}
d^n\varphi(e_0,\dots,e_n)&=\sum\limits_{i=0}^{n+1}(-1)^{i}x_i.\varphi(e_0,\dots,\hat{e}_i,\dots,e_n)\\[3mm]
&+\sum\limits_{1\leq i<j\leq n}(-1)^{i+j}\varphi([e_i,e_j],e_0,\dots,\hat{e}_i,\dots,\hat{e}_j,\dots,e_n),
\end{array}
\end{equation}
 where $e_0,\dots,e_n\in\mathfrak{g}$ and the sign $\hat{}$ indicates that the argument below it must be omitted. Set $Z^{n}(\mathfrak{g},M):={\rm Ker}(d^{n+1})$ and  $B^{n}(\mathfrak{g},M):={\rm Im}(d^{n}).$ The property $d^{n+1}\circ d^{n}=0$ implies that the total differential $d=\sum\limits_{i\geq0}d^i$ satisfies $d\circ d=0.$ Therefore, the $n$-th cohomology space
$$H^{n}(\mathfrak{g},M)=Z^{n}(\mathfrak{g},M)/B^{n}(\mathfrak{g},M)$$
is well-defined.

The elements of $Z^{n}(\mathfrak{g},M)$ and $B^{n}(\mathfrak{g},M)$ are called $n$-{\it cocycles} and $n$-{\it coboundaries}, respectively.
Note that, the second cocycles with values in the adjoint representation are called \textit{infinitesimal deformations}.
The central extensions of $\mathfrak{g}$ are classified by the second cohomology group with values in the trivial representation.

If $\mathfrak{r}$ is an ideal of $\mathfrak{g},$ then we define on each $C^n(\mathfrak{r},M)$ the structure of $\mathfrak{g}$-module. Because of $C^{0}(\mathfrak{r},M)=M,$ the $\mathfrak{g}$-action on $M$ defines $\mathfrak{g}$-module structure on $C^{0}(\mathfrak{r},M)$. For $n>0$, we define the action $\mathfrak{g}$ on  $C^n(\mathfrak{r},M)$ as follow:
\begin{align}{\label{action}}
(v.\omega)(e_1,\dots,e_n)=v.\omega(e_1,\dots,e_n)-\sum_{i=1}^{n}\omega(e_1,\dots,e_{i-1},[v,e_{i}],e_{i+1},\dots,e_n),
\end{align}
where $v \in \mathfrak{g},$ $\omega \in C^n(\mathfrak{r},M)$ and $e_1,\dots,e_n\in\mathfrak{r}$.

In general, the computation of cohomology spaces is complicated. However, if a Lie algebra is the semidirect product of a semisimple subalgebra and its radical, the Hochschild-Serre factorization theorem (Theorem 13 in \cite{HchSe}) simplifies the computation. Let $\mathfrak{g}=\mathfrak{s}\ltimes\mathfrak{r}$ be a Lie algebra, where $\mathfrak{s}$ is the semisimple part of $\mathfrak{g}$, and  $\mathfrak{r}$ is its radical, $M$ be a finite dimensional $\mathfrak{g}$-module. Then $H^p(\mathfrak{g},M)$ satisfies the following isomorphism of vector spaces
$$H^p(\mathfrak{g},M)\cong\sum_{m+n=p}H^m(\mathfrak{s},\mathbb{C})\otimes H^n(\mathfrak{r},M)^{\mathfrak{s}}.$$

Here, $H^d(\mathfrak{r},M)^{\mathfrak{s}}$ is the space of $\mathfrak{s}$-invariant cocycles of $\mathfrak{r}$ with values in $M$ which is defined as
$$H^d(\mathfrak{r},M)^{\mathfrak{s}}=\{[\varphi]\in H^d(\mathfrak{r},M) \ | \ \forall v\in\mathfrak{s}:  v.\varphi\in B^d(\mathfrak{r},M)\},$$
where $[\varphi]$ denotes the cohomology class of the $d$-cocycle $\varphi$ and the action $v.\varphi$ defined as \eqref{action}. Namely, since the $\mathfrak{g}$-action on $C^d(\mathfrak{r},M)$, as defined in \eqref{action}, is compatible with the differential, $Z^d(\mathfrak{r},M)$ and $B^d(\mathfrak{r},M)$ are invariant under the action of $\mathfrak{s}$, and therefore this action can be lifted to $H^d(\mathfrak{r},M)$.
Note that the following isomorphism holds for $H^d(\mathfrak{r},M)^{\mathfrak{s}}$:
\begin{align}{\label{iso}}
\nonumber
H^d(\mathfrak{r},M)^{\mathfrak{s}}&=[Z^d(\mathfrak{r},M)^{\mathfrak{s}}+B^d(\mathfrak{r},M)/B^d(\mathfrak{r},M)]\\
                                            &\simeq Z^d(\mathfrak{r},M)^{\mathfrak{s}}/[B^d(\mathfrak{r},M)\cap Z^d(\mathfrak{r},M)^{\mathfrak{s}}]\\
                                            \nonumber
                                            &= Z^d(\mathfrak{r},M)^{\mathfrak{s}}/B^d(\mathfrak{r},M)^{\mathfrak{s}}.
\end{align}

The first equality is based on the following fact: Since $Z^d(\mathfrak{r},M)$ is finite-dimensional, $Z^d(\mathfrak{r},M)$ is completely reducible as $\mathfrak{s}$-module. Therefore, for any cocycle $\varphi\in Z^d(\mathfrak{r},M)$ there exists an $\mathfrak{s}$-invariant cocycle $\varphi^{\prime}$ such that $\varphi-\varphi^{\prime}\in B^d(\mathfrak{r},M)$ (see Proposition 2.1 in \cite{WuZh}).

Now, let us introduce the $n$-th Schr\"{o}dinger algebra. Recall that the special linear Lie algebra $\mathfrak{sl}_2$ of traceless $2\times 2$ matrices is a simple Lie algebra with the following products for the standard basis $\{e,h,f\}$:
\begin{align*}
[e,f]=h, \quad  [h,e]=2e, \quad  [h,f]=-2f.
\end{align*}

It is known that general linear algebra $\mathfrak{gl}_{2n}$ has the natural representation on $\mathbb{C}^{2n}$ by left matrix multiplication. Let $\{e_1,e_2,\dots,e_{2n}\}$ be the standard basis of $\mathbb{C}^{2n}$.
 The Heisenberg algebra $\mathfrak{h}_n=\mathbb{C}^{2n}\oplus\mathbb{C}z$ is the Lie algebra with the following multiplication table:
$$[e_i,e_{n+i}]=z, \quad [z,\mathfrak{h}_{n}]=0.$$

The $n$-th Schr\"{o}dinger algebra $\mathfrak{sch}_n$ is the semidirect product of $\mathfrak{sl}_2$ and a Heisenberg algebra $\mathfrak{h}_n.$  Here, $\mathfrak{sl}_2$ is embedded in $\mathfrak{gl}_{2n}$ by the mapping
$$\begin{pmatrix}
a & b \\
c & -a
\end{pmatrix}\rightarrow
\begin{pmatrix}
aI_n & bI_n \\
cI_n & -aI_n
\end{pmatrix},$$
where $I_n$ is the $n\times n$ identity matrix, $\mathfrak{sl}_2$ acts on $\mathfrak{h}_n$ by matrix multiplication and $[z,\mathfrak{sch}_{n}]=0$. Now, we denote
$$h=
\begin{pmatrix}
I_n & 0 \\
0 & -I_n
\end{pmatrix}, \quad e=
\begin{pmatrix}
0 & I_n \\
0 & 0
\end{pmatrix}, \quad f=
\begin{pmatrix}
0 & 0 \\
I_n & 0
\end{pmatrix}, \quad x_k=e_k, \quad y_k=e_{n+k},$$
where $1\leq k\leq n$. Then, one can verify directly that the $n$-th Schr\"{o}dinger algebra $\mathfrak{sch}_n$ is a Lie algebra with a basis $\{e,h,f,x_i,y_i, \ 1\leq i\leq n\}$ equipped with the following non-trivial commutation relations:
\begin{align*}
[h,e]&=2e, &  [h,f]&=-2f & [e,f]&=h,\\
[h,x_i]&=x_i, & [h,y_i]&=-y_i,   &    [e,y_i]&=x_i, & [f,x_i]&=y_i,\\
[x_i,y_i]&=z.
\end{align*}

Thanks to \cite{LY}, we have the following description of the space of derivations of the algebra $\mathfrak{sch}_n$:
$$\mathrm{Der}(\mathfrak{sch}_n)=\mathrm{Inn}(\mathfrak{sch}_n)  \oplus \bigoplus\limits_{1\leq i<j\leq n} \mathbb{C}\sigma_{ij}\bigoplus\mathbb{C}\tau.$$
where the outer derivations $\sigma_{ij}$ and  $\tau$ are defined as follows:
\begin{align*}
        \sigma_{ij}(h)&=\sigma_{ij}(e)=\sigma_{ij}(f)=\sigma_{ij}(z)=0,\\
        \sigma_{ij}(x_k)&=\delta_{ik}x_j-\delta_{jk}x_i,\qquad \sigma_{ij}(y_k)=\delta_{ik}y_j-\delta_{jk}y_i
\end{align*}
for $1\leq i<j\leq n$ and
\begin{align*}
       \tau(h)&=\tau(e)=\tau(f)=0,\ \tau(u_i)=u_i,\ \tau(v_i)=v_i,\ \tau(z)=2z,\  1\leq i\leq n.
\end{align*}

It should be noted that the difference between the $n$-th Schr\"{o}dinger algebra $\mathfrak{sch}_n$
and the Schr\"odinger algebra in $(n+1)$-dimensional space-time is that the semisimple part of the latter contains the orthogonal Lie algebra $\mathfrak{so}_{n}.$
In \cite{Ru}, the Schr\"odinger algebra in $(n+1)$-dimensional space-time and its factor algebra by one-dimensional center are denoted by $\widehat{S}(N)$ and $S(N),$ respectively.
The second cohomology groups of these algebras with values in the adjoint modules are computed, and it is proven that $$H^2(\widehat{S}(N), \widehat{S}(N)) = H^2(S(N), S(N))=0, \ \text{for} \  N \geq 3,$$ $$\dim H^2(\widehat{S}(2), \widehat{S}(2)) = \dim H^2(S(2), S(2))=2.$$

Results of the present work are similar to those in \cite{Ru}, specifically
we compute the second cohomology groups for the $n$-th Schr\"{o}dinger algebra $\mathfrak{sch}_n$ and its factor algebra by one-dimensional center with values in the adjoint module.
Moreover, we determine the second cohomology groups of these algebras with values in the trivial representation.

\section{Infinitesimal deformations and central extensions of the $n$-th Schr\"{o}dinger algebra}

\subsection{Central extensions of the $n$-th Schr\"{o}dinger algebra}
In this subsection, we compute the dimension of the second cohomology group of the $n$-th Schr\"{o}dinger algebra with values in the trivial representations.
In this case the first and second
 differentials \eqref{eq1} are given by
\begin{align}{\label{Z1trivial}} d^1\omega(x,y)=\omega([x,y])\end{align}
and
\begin{align}{\label{Z2trivial}}
d^2\varphi(x,y,z) = \varphi(x,[y,z])+\varphi(y,[z,x])+\varphi(z,[x,y]).
\end{align}


Note that the second cohomology group for $n=1$ was determined in \cite{WuZh}, and it is shown that $ \dim(H^2(\mathfrak{sch}_1,\mathbb{C}))=0$.
Therefore, we consider the case of $n\geq2$.

\begin{thm}{\label{thm1}}
$ \dim(H^2(\mathfrak{sch}_n,\mathbb{C}))=\frac{n(n+1)}{2}-1$ for $n\geq2.$
\end{thm}

\begin{proof}
Since $\mathbb{C}$ is a trivial $\mathfrak{sl}_2$-module, we have
  $H^0(\mathfrak{sl}_2,\mathbb{C})\simeq\mathbb{C}.$ By Whitehead's first and second lemmas, we get
  $H^1(\mathfrak{sl}_2,\mathbb{C})=H^2(\mathfrak{sl}_2,\mathbb{C})=0$, and therefore the Hochschild-Serre factorization theorem (Theorem 13 in \cite{HchSe}) implies that
$$H^2(\mathfrak{sch}_n,\mathbb{C})\simeq H^2(\mathfrak{h}_n,\mathbb{C})^{\mathfrak{sl}_2}.$$

From the multiplication of the Heisenberg algebra $\mathfrak{h}_n,$ we easily get that $B^2(\mathfrak{h}_n,\mathbb{C})^{\mathfrak{sl}_2}=\langle \phi\rangle,$ where $\phi(x_i,y_i)=\omega([x_i, y_i]) = \omega(z)$ for some $\omega\in{\rm Hom}(\mathfrak{h}_n,\mathbb{C})$. Thus, $\dim(B^2(\mathfrak{h}_n,\mathbb{C})^{\mathfrak{sl}_2})=1.$

Let $\varphi\in Z^2(\mathfrak{h}_n,\mathbb{C})^{\mathfrak{sl}_2}.$ 
Considering $h.\varphi(x_i,z)=0$ and $h.\varphi(y_i,z)=0,$ we obtain $\varphi(x_i, z) = 0$ and $\varphi(y_i, z)=0$ for $i\in\{1,2,\dots,n\}.$
Next, from $h.\varphi(x_i,x_j)=0$ and $h.\varphi(y_i,y_j)=0,$ we get $\varphi(x_i, x_j) = 0$ and $\varphi(y_i, y_j)=0$ for any $i, j \in\{1,2,\dots,n\}.$
Moreover, the equality $$0=e.\varphi(y_i,y_j)= \varphi([e,y_i], y_j)+\varphi(y_i, [e,y_j])=\varphi(x_i, y_j)+\varphi(y_i, x_j)$$
implies that $\varphi(x_i, y_j)=\varphi(x_j, y_i)$ for $i\neq j.$
Summarizing the relations above, we deduce that $\dim(Z^2(\mathfrak{h}_n,\mathbb{C})^{\mathfrak{sl}_2})=\frac{n(n+1)}{2}$. Thus, by \eqref{iso} we have
$H^2(\mathfrak{h}_n,\mathbb{C})^{\mathfrak{sl}_2}\simeq Z^2(\mathfrak{h}_n,\mathbb{C})^{\mathfrak{sl}_2}/B^2(\mathfrak{h}_n,\mathbb{C})^{\mathfrak{sl}_2}.$ Hence, $ \dim(H^2(\mathfrak{sch}_n,\mathbb{C}))=\frac{n(n+1)}{2}-1$.


\end{proof}

Now, let $\varphi_{ij}$ $(i,j\in\{1,2,\dots,n\})$ be anti-symmetric bilinear forms defined as follows:
\begin{align}
 \varphi_{ij}(x_k,y_l)
 {\label{fij2}}
&=\left\{
\begin{array}{ccclll}
1, & \mbox{ if } \ (i,j)=(k,l) \mbox{ or } (i,j)=(l,k)\\
0, & \mbox{ otherwise},
\end{array}
 \right. \\
{\label{fij}}
\varphi_{ij}(x_k,x_l)&=\varphi_{ij}(y_k,y_l)=0,\\
{\label{fij0}}
\varphi_{ij}(x_k,z)&=\varphi_{ij}(y_k,z)=0.
 \end{align}

From the proof of Theorem \ref{thm1}, it is not difficult to obtain the following corollary.

\begin{cor}
The set of anti-symmetric bilinear forms $\varphi_{ij}\ (\ i,j\in\{1,2,\dots,n\})$ satisfying  \eqref{fij2}, \eqref{fij} and \eqref{fij0} form a basis for $Z^2(\mathfrak{sch}_n,\mathbb{C})$.
\end{cor}


\subsection{Infinitesimal deformations of the $n$-th Schr\"{o}dinger algebra}

In this subsection, we  determine the second cohomology group of $\mathfrak{sch}_n$ with values in the adjoint
representation. In \cite{WuZh}, it is proven that if $V$ is a finite-dimensional irreducible $\mathfrak{sch}_1$-module and $\dim (V)\neq 2$, then $H^2(\mathfrak{sch}_1,V)=0$. The algebra $\mathfrak{sch}_1$ is itself an irreducible $\mathfrak{sch}_1$-module and $\dim(\mathfrak{sch}_1)=6$, and therefore $H^2(\mathfrak{sch}_1,\mathfrak{sch}_1)=0$.
Thus, we consider the case of $n \geq 2.$

For the adjoint representation the first and second differentials \eqref{eq1} are given by
\begin{align}{\label{eq4.1}}d^1\omega(x,y)=[\omega(x),y]+[x,\omega(y)]-\omega([x,y])\end{align}
and
\begin{align}{\label{eq4.2}}
d^2\varphi(x,y,z) = [x,\varphi(y,z)]+[y,\varphi(z,x)]+[z,\varphi(x,y)]+\varphi(x,[y,z])+\varphi(y,[z,x])+\varphi(z,[x,y]).
\end{align}

Thus, the linear space of second cocycles $Z^2(\mathfrak{g},\mathfrak{g})$ is the set of alternating bilinear maps satisfying the condition $d^2\varphi(x,y,z)=0$.
The subspace $B^2(\mathfrak{g},\mathfrak{g})$ is defined as follows:
$$B^2(\mathfrak{g},\mathfrak{g})=\{\phi \ | \ \exists \omega\in C^1(\mathfrak{g},\mathfrak{g}): \ \phi(x,y)=d^1\omega(x,y)\}.$$


\begin{thm}{\label{mainthm}}
$ H^2(\mathfrak{sch}_n,\mathfrak{sch}_n)=0$ for $n\geq3.$
\end{thm}

By \eqref{action}, we have
\begin{align} \label{eq4.3} v.\varphi(e_1,e_2) = [v, \varphi(e_1,e_2)] - \varphi([v, e_1], e_2) - \varphi(e_1, [v,e_2]) \quad \text{for} \quad v\in\mathfrak{sl}_2, \ e_1, e_2 \in\mathfrak{h}_n.
\end{align}


At first, let us prove the following auxiliary lemmas:

\begin{lem}{\label{mainlem1}}
$ \dim(B^2(\mathfrak{h}_n,\mathfrak{sch}_n)^{\mathfrak{sl}_2})=\frac{n(n+1)}{2}$ for $n\geq2.$
\end{lem}

\begin{proof}

Let $\omega\in C^{1}(\mathfrak{h}_n,\mathfrak{sch}_n),$ then the action \eqref{action} has the form
\begin{align}
v.\omega(x)&=[v, \omega(x)] - \omega([v,x]) \quad \text{for} \quad v \in \mathfrak{sl}_2, x \in \mathfrak{h}_n.
\end{align}

First, considering $v.\omega(z)=0$ for any $v\in \mathfrak{sl}_2,$ we have
$$0=[v,\omega(z)]-\omega([v,z])=[v,\omega(z)].$$

Taking into account that $Z(\mathfrak{sl}_2)=0$ and $Z(\mathfrak{sch}_n)=\langle z\rangle$, we obtain $\omega(z)\in\langle z\rangle.$

Now, applying the argument $h.\omega(x_i)=0$ on the elements $h\in\mathfrak{sl}_2,$ $x_i\in\mathfrak{h}_n$, we get
$$\omega(x_i)=\omega([h,x_i])=[h,\omega(x_i)].$$

Thus, $\omega(x_i)$ is an eigenvector for the operator $\operatorname{ad}_h$ with eigenvalue 1. Such
eigenvectors are $x_1,x_2,\dots,x_n.$ Hence, we obtain $\omega(x_i)\in\langle x_1,x_2,\dots,x_n\rangle$ for all $i\in\{1,2,\dots,n\}.$
Similarly, considering  $h.\omega(y_i)=0,$ we obtain
$\omega(y_i)=-\omega([h,y_i])=-[h,\omega(y_i)],$ which implies $\omega(y_i)\in\langle y_1,y_2,\dots,y_n\rangle$ for all $i\in\{1,2,\dots,n\}.$
Therefore, we can set
\begin{align} \label{eq3.7}
\omega(x_i)=\sum_{j=1}^{n}\alpha_{i,j}x_j,  \quad  \omega(y_i)=\sum_{j=1}^{n}\beta_{i,j}y_j, \quad \omega(z)=\gamma z.
\end{align}

From $e.\omega(y_i)=0$, we have
\begin{align*}\omega(x_i)&=\omega([e,y_i])=[e,\omega(y_i)]=\sum_{j=1}^{n}\beta_{i,j}[e,y_j]=\sum_{j=1}^{n}\beta_{i,j}x_j,\end{align*}
which implies  \begin{align} \label{eq3.8}\alpha_{i,j}=\beta_{i,j}.\end{align}

Thus, from \eqref{eq3.7} and \eqref{eq3.8} we can conclude that $\dim(C^{1}(\mathfrak{h}_n,\mathfrak{sch}_n)^{\mathfrak{sl}_2})=n^2+1$.
However, $\sigma_{i,j}, \ 1\leq i<j\leq n$ and $\tau$ are the only derivations which satisfy \eqref{eq3.8} and are described in \cite{LY}. Therefore, we have $\dim({\rm Der}(\mathfrak{h}_n,\mathfrak{sch}_n)^{\mathfrak{sl}_2})=\frac{n(n-1)}{2}+1$. Hence,
$$\dim (B^2(\mathfrak{h}_n,\mathfrak{sch}_n)^{\mathfrak{sl}_2})=\dim(C^{1}(\mathfrak{h}_n,\mathfrak{sch}_n)^{\mathfrak{sl}_2})-\dim({\rm Der}(\mathfrak{h}_n,\mathfrak{sch}_n)^{\mathfrak{sl}_2})=\frac{n(n+1)}{2}.$$
\end{proof}

\begin{lem}{\label{mainlem2}}
$ \dim(Z^2(\mathfrak{h}_n,\mathfrak{sch}_n)^{\mathfrak{sl}_2})=\frac{n(n+1)}{2}$ for $n\geq3.$
\end{lem}

\begin{proof}
For the central element $z\in\mathfrak{h}_n,$ define a linear mapping $\phi:\mathfrak{h}_n\rightarrow \mathfrak{sch}_n$ as follows:
$$\phi(-)=\varphi(-,z).$$

 Then the $\mathfrak{sl}_2$-invariance of $\varphi$ under the action \eqref{eq4.3}
 implies that
\begin{align*}
[v,\varphi(a,z)]=\varphi([v,a],z) \ \mbox{ or } \ [v,\phi(a)]=\phi([v,a]),
\end{align*}
where $v\in\mathfrak{sl}_2,$ $a\in\mathfrak{h}_n$ and $z$ is the central element. In other words, $\phi$ is invariant with respect to the subalgebra $\mathfrak{sl}_2$.
Thus, according to \eqref{eq3.7} and \eqref{eq3.8},
we conclude that
\begin{align*}
\phi(x_i)=\sum_{j=1}^{n}\alpha_{i,j}x_j,  \quad  \phi(y_i)=\sum_{j=1}^{n}\alpha_{i,j}y_j, \quad \phi(z)=\gamma z.
\end{align*}

Hence,\begin{align*}
\varphi(x_i,z)=\sum_{j=1}^{n}\alpha_{i,j}x_j,  \quad  \varphi(y_i,z)=\sum_{j=1}^{n}\alpha_{i,j}y_j.
\end{align*}

Now, considering the $\mathfrak{sl}_2$-invariance of $\varphi$ with respect to $h,$ i.e.,  $h.\varphi(x_i,x_j)=0,$ we obtain
$$[h,\varphi(x_i,x_j)]=\varphi([h,x_i],x_j)+\varphi(x_i,[h,x_j])=2\varphi(x_i,x_j).$$

Then we conclude that elements $\varphi(x_i,x_j)$ are eigenvectors of the operator $\operatorname{ad}_h$ with eigenvalue $2.$ Since, such eigenvector of $\operatorname{ad}_h$ is only $e,$ we derive that $\varphi(x_i,x_j)\in\langle e\rangle.$
Similarly, one can show that $\varphi(y_i,y_j)=\langle f\rangle,$ $i,j\in\{1,2,\dots,n\}.$

Next, applying the condition $h.\varphi(x_i,y_j)=0,$ we have
$$[h,\varphi(x_i,y_j)]=\varphi([h,x_i],y_j)+\varphi(x_i,[h,y_j])=0.$$

Hence, we can set  $\varphi(x_i,y_j)=\beta_{i,j}h+\gamma_{i,j}z$. Then, from
$$-2\beta_{i,j}e=[e,\varphi(x_i,y_j)]=\varphi([e,x_i],y_j)+\varphi(x_i,[e,y_j])=\varphi(x_i,x_j),$$
we obtain $\varphi(x_i,x_j)=-2\beta_{i,j}e.$ By similar computations, we can obtain that $\varphi(y_i,y_j)=2\beta_{i,j}f.$
Further, applying the condition $e.\varphi(y_i,y_j)=0,$ we have
$$[e,\varphi(y_i,y_j)]=\varphi([e,y_i],y_j)+\varphi(y_i,[e,y_j])=\varphi(x_i,y_j)+\varphi(y_i,x_j),$$
which follows $\gamma_{i,j}=\gamma_{j,i}.$ Furthermore, the relation $f.\varphi(x_i,y_i)=0$ implies that
$$2\beta_{i,i}f=[f,\varphi(x_i,y_i)]=\varphi([f,x_i],y_i)+\varphi(x_i,[f,y_i])=\varphi(y_i,y_i)=0.$$
Hence, $\beta_{i,i}=0$ that is $\varphi(x_i,y_i)=\gamma_{i,i}z$.

Now, we use the condition $d^2\varphi(x,y,z)=0$ for $x,y,z\in\mathfrak{h}_n$.
Then, for $x_i,y_j,y_k\in\mathfrak{h}_n$ we have
\begin{align*}
0&=d^2(x_i,y_j,y_k)\\
&=[x_i,\varphi(y_j,y_k)]+[y_j,\varphi(y_k,x_i)]+[y_k,\varphi(x_i,y_j)]+\varphi(x_i,[y_j,y_k])+\varphi(y_j,[y_k,x_i])+\varphi(y_k,[x_i,y_j])\\
&=2\beta_{j,k}[x_i,f]-\beta_{i,k}[y_j,h]+\beta_{i,j}[y_k,h]\\
&=-2\beta_{j,k}y_i-\beta_{i,k}y_j+\beta_{i,j}y_k.
\end{align*}
Thus, we obtain $\beta_{i,j}=0$ for $1\leq i\neq j\leq n$, that is, $\varphi(x_i,x_j)=0$ for $i\neq j$.
Moreover, for $x_i,y_i,y_j\in\mathfrak{h}_n$ we have
\begin{align*}
0&=d^2(x_i,y_i,y_j)\\
&=[x_i,\varphi(y_i,y_j)]+[y_i,\varphi(y_j,x_i)]+[y_j,\varphi(x_i,y_i)]+\varphi(x_i,[y_i,y_j])+\varphi(y_i,[y_j,x_i])+\varphi(y_j,[x_i,y_i]),\\
&=\varphi(y_j,z),
\end{align*}
which follows that  $\varphi(y_j,z)=0$ for $i,j\in\{1,2,\dots,n\}$.
Similarly, one can show that $\varphi(x_j,z)=0$ for $1\leq j\leq n$.
Therefore, we obtain that
\begin{align*}
\varphi(x_i,y_j)&=\gamma_{i,j} z, & \varphi(x_i,x_j)&=0, & \varphi(y_i,y_j)&=0,\\
\varphi(x_i,z)&=0, & \varphi(y_i,z)&=0.
\end{align*}
Thus, $\dim(Z^2(\mathfrak{h}_n,\mathfrak{sch}_n)^{\mathfrak{sl}_2})=\frac{n(n+1)}{2}.$
\end{proof}


\begin{proof}[Proof of Theorem \ref{mainthm}] By the Hochschild-Serre factorization theorem (Theorem 13 in \cite{HchSe}), we have
$$H^2(\mathfrak{sch}_n,\mathfrak{sch}_n)\simeq
H^0(\mathfrak{sl}_2,\mathbb{C})\otimes H^2(\mathfrak{h}_n,\mathfrak{sch}_n)^{\mathfrak{sl}_2}\oplus
H^1(\mathfrak{sl}_2,\mathbb{C})\otimes H^1(\mathfrak{h}_n,\mathfrak{sch}_n)^{\mathfrak{sl}_2}\oplus
H^2(\mathfrak{sl}_2,\mathbb{C})\otimes H^0(\mathfrak{h}_n,\mathfrak{sch}_n)^{\mathfrak{sl}_2}.$$
Taking $H^0(\mathfrak{sl}_2,\mathbb{C})\simeq\mathbb{C}, \ H^1(\mathfrak{sl}_2,\mathbb{C})=H^2(\mathfrak{sl}_2,\mathbb{C})=0$ into account we get
$$H^2(\mathfrak{sch}_n,\mathfrak{sch}_n)\simeq
\mathbb{C}\otimes H^2(\mathfrak{h}_n,\mathfrak{sch}_n)^{\mathfrak{sl}_2}\simeq H^2(\mathfrak{h}_n,\mathfrak{sch}_n)^{\mathfrak{sl}_2}.$$

On the other hand, \eqref{iso} follows that $H^2(\mathfrak{h}_n,\mathfrak{sch}_n)^{\mathfrak{sl}_2} \simeq  Z^2(\mathfrak{h}_n,\mathfrak{sch}_n)^{\mathfrak{sl}_2}/B^2(\mathfrak{h}_n,\mathfrak{sch}_n)^{\mathfrak{sl}_2}.$
Moreover, Lemma \ref{mainlem1} and \ref{mainlem2} imply that
$$ \dim(H^2(\mathfrak{h}_n,\mathfrak{sch}_n)^{\mathfrak{sl}_2})
=\dim(Z^2(\mathfrak{h}_n,\mathfrak{sch}_n)^{\mathfrak{sl}_2})-\dim(B^2(\mathfrak{h}_n,\mathfrak{sch}_n)^{\mathfrak{sl}_2})=0.$$

Hence, $ H^2(\mathfrak{sch}_n,\mathfrak{sch}_n)=0$.
\end{proof}

By Theorem 7.1 in \cite{NR} one can obtain the following corollary:

\begin{cor}{\label{pr}} The $n$-th Schr\"{o}dinger algebra $\mathfrak{sch}_n$ is rigid for  $ n \geq 3.$
\end{cor}

Now, we consider the case $n=2.$

\begin{prop}{\label{prop}} $ \dim(H^2(\mathfrak{sch}_2,\mathfrak{sch}_2))=1.$
\end{prop}

\begin{proof}

Let $\varphi\in Z^2(\mathfrak{h}_2,\mathfrak{sch}_2)^{\mathfrak{sl}_2}.$ From
\begin{align*}
[h,\varphi(x_i, z)]&=\varphi([h,x_i],z)+\varphi(x_i,[h,z])=\varphi(x_i,z),\\
[h,\varphi(y_i, z)]&=\varphi([h,y_i],z)+\varphi(y_i,[h,z])=-\varphi(y_i,z),
\end{align*}
we obtain that $\varphi(x_i,z)\in\langle x_1,x_2\rangle$ and $\varphi(y_i,z)\in\langle y_1,y_2\rangle$ for $i=1,2.$
Moreover, from
\begin{align*}
[h,\varphi(x_1,x_2)]&=\varphi([h,x_1],x_2)+\varphi(x_1,[h,x_2])=2\varphi(x_1,x_2),\\
[h,\varphi(y_1,y_2)]&=\varphi([h,y_1],y_2)+\varphi(y_1,[h,y_2])=-2\varphi(y_1,y_2),
\end{align*}
we can easily get that $\varphi(x_1,x_2)\in\langle e\rangle$ and  $\varphi(y_1,y_2)\in\langle f\rangle.$

Furthermore, from
\begin{align*}[h,\varphi(x_1,y_2)]=\varphi([h,x_1],y_2)+\varphi(x_1,[h,y_2])=\varphi(x_1,y_2)-\varphi(x_1,y_2)=0,\\
[h,\varphi(x_2,y_1)]=\varphi([h,x_2],y_1)+\varphi(x_2,[h,y_1])=\varphi(x_2,y_1)-\varphi(x_2,y_1)=0,\end{align*}
one can verify that $\varphi(x_1,y_2), \varphi(x_2,y_1)\in\langle h,z\rangle.$

Finally, by the relations
 \begin{align*}
 [e,\varphi(x_i,y_i)]&=\varphi([e,x_i],y_i)+\varphi(x_i,[e,y_i])=\varphi(x_i,x_i)=0,\\
 [f,\varphi(x_i,y_i)]&=\varphi([f,x_i],y_i)+\varphi(x_i,[f,y_i])=\varphi(y_i,y_i)=0,
 \end{align*}
we obtain $\varphi(x_1,y_1),\varphi(x_2,y_2)\in\langle z\rangle.$

Therefore, we can set
\begin{align*}
\varphi(x_1,x_2)&=\alpha_1 e, & \varphi(x_1,y_1)&=\alpha_2 z, & \varphi(x_1,y_2)&=\alpha_3 h+ \alpha_4z,\\
\varphi(y_1,y_2)&=\beta_1 f, & \varphi(x_2,y_2)&=\beta_2 z,& \varphi(x_2,y_1)&=\beta_3 h+ \beta_4z, \\
\varphi(x_1,z)&=\gamma_1 x_1+\gamma_2 x_2, & \varphi(y_1, z)&=\mu_1 y_1+\mu_2 y_2, & &\\
\varphi(x_2,z)&=\gamma_3 x_1+\gamma_4 x_2, & \varphi(y_2, z)&=\mu_3 y_1+\mu_4 y_2, & &
\end{align*}

Moreover, taking $e.\varphi(y_i,z)=0$ into account, we can get
 $\mu_j = \gamma_j$ for $1 \leq j \leq 4.$
On the other hand, from $e.\varphi(x_1,y_2)=0$ and $f.\varphi(x_1,y_2)=0,$ we obtain $\alpha_1 = -2\alpha_3,$ $\beta_1 = 2\alpha_3.$
Next, considering
$$[f,\varphi(x_1,x_2)]=\varphi([f,x_1],x_2)+\varphi(x_1,[f,x_2])=\varphi(y_1,x_2)+\varphi(x_1,y_2),$$
we derive $\alpha_1=\beta_3-\alpha_3$ and $\beta_4=\alpha_4.$ Hence, we get $\beta_3=-\alpha_3$.

Now, we use the condition that $d^2\varphi(a, b, c) =0$ for $a, b, c \in \mathfrak{h}_2.$
First, considering $d^2\varphi(x_1,y_1,y_2)=0,$ we have
$$\varphi(x_1,[y_1,y_2])+\varphi(y_1,[y_2,x_1])+\varphi(y_2,[x_{1},y_1])+[x_{1},\varphi(y_1,y_2)]+[y_1,\varphi(y_2,x_{1})]+[y_2,\varphi(x_{1},y_1)]=0,$$
which follows
\[\varphi(y_2,z)+\beta_1[x_1,f]-\alpha_3[y_1,h]=0.\]
Thus, we get $\varphi(y_2,z)=(\beta_1+\alpha_3)y_1.$ Similarly, considering $d^2\varphi(x_2,y_1,y_2)=0,$ we obtain $\varphi(y_1,z)=-(\beta_1+\alpha_3)y_2.$ Moreover, $d^2\varphi(x_1,x_2,y_2)=0$ and $d^2\varphi(x_1,x_2,y_1)=0$ imply that $\varphi(x_1,z)=(\alpha_1+\beta_3)x_2$ and $\varphi(x_2,z)=-(\alpha_1+\beta_3)x_1$, respectively.
From all this and the substitutions $\alpha_2=\alpha, \ \beta_2=\beta, \ \alpha_3=\gamma$, $\alpha_4=\mu$, we conclude that $\varphi$ has the following form:
\begin{align*}
\varphi(x_1,x_2)&=-2\gamma e, & \varphi(x_1,y_1)&=\alpha z, & \varphi(x_1,y_2)&=\gamma h+\mu z,\\
\varphi(y_1,y_2)&=2\gamma f, & \varphi(x_2,y_2)&=\beta z,& \varphi(x_2,y_1)&=-\gamma h+ \mu z,\\
\varphi(x_1,z)&=- 3\gamma x_2, & \varphi(y_1, z)&=- 3\gamma y_2, & &\\
\varphi(x_2,z)&= 3\gamma x_1, & \varphi(y_2, z)&= 3\gamma y_1. & &
\end{align*}

Therefore, $\dim(Z^2(\mathfrak{h}_2,\mathfrak{sch}_2)^{\mathfrak{sl}_2})=4.$ It follows from Lemma \ref{mainlem1} that $\dim(B^2(\mathfrak{h}_2,\mathfrak{sch}_2)^{\mathfrak{sl}_2})=3.$
Then, by \eqref{iso}, we have $H^2(\mathfrak{h}_2,\mathfrak{sch}_2)^{\mathfrak{sl}_2}\simeq Z^2(\mathfrak{h}_2,\mathfrak{sch}_2)^{\mathfrak{sl}_2}/B^2(\mathfrak{h}_2,\mathfrak{sch}_2)^{\mathfrak{sl}_2}$.
So, we get $$\dim(H^2(\mathfrak{h}_2,\mathfrak{sch}_2)^{\mathfrak{sl}_2})=1.$$

By the Hochschild-Serre factorization theorem (Theorem 13 in \cite{HchSe}), we have
$$H^2(\mathfrak{sch}_2,\mathfrak{sch}_2)\simeq
H^0(\mathfrak{sl}_2,\mathbb{C})\otimes H^2(\mathfrak{h}_2,\mathfrak{sch}_2)^{\mathfrak{sl}_2}\oplus
H^1(\mathfrak{sl}_2,\mathbb{C})\otimes H^1(\mathfrak{h}_2,\mathfrak{sch}_2)^{\mathfrak{sl}_2}\oplus
H^2(\mathfrak{sl}_2,\mathbb{C})\otimes H^0(\mathfrak{h}_2,\mathfrak{sch}_n)^{\mathfrak{sl}_2}.$$

Taking $H^0(\mathfrak{sl}_2,\mathbb{C})\simeq\mathbb{C}, \ H^1(\mathfrak{sl}_2,\mathbb{C})=H^2(\mathfrak{sl}_2,\mathbb{C})=0$ into account we get
$$H^2(\mathfrak{sch}_2,\mathfrak{sch}_2)\simeq
\mathbb{C}\otimes H^2(\mathfrak{h}_2,\mathfrak{sch}_2)^{\mathfrak{sl}_2}\simeq H^2(\mathfrak{h}_2,\mathfrak{sch}_2)^{\mathfrak{sl}_2}.$$

Thus, $\dim(H^2(\mathfrak{sch}_2,\mathfrak{sch}_2))=1.$
\end{proof}

\begin{cor} The alternating bilinear map $\psi$, defined as below, forms a basis on $H^2(\mathfrak{sch}_2,\mathfrak{sch}_2).$
\begin{align*}
\psi(x_1,x_2)&=2e, &  \psi(y_1,y_2)&=-2f,\\
\psi(x_1,y_2)&=-h,  & \psi(x_2,y_1)&=h,\\
\varphi(x_1,z)&= 3 x_2,& \varphi(y_1, z)&=3 y_2,\\
\varphi(x_2,z)&= -3 x_1, & \varphi(y_2, z)&= -3 y_1.
\end{align*}

\end{cor}

\section{Infinitesimal deformations and central extensions of $\mathfrak{sch}_n / \langle z \rangle$}

Let $\mathfrak{g}_n$ be the factor algebra of $\mathfrak{sch}_n$ by its center, i.e., $\mathfrak{g}_n = \mathfrak{sch}_n / \langle z \rangle.$ Then the brackets of $\mathfrak{g}_n$ are as follows:
$$\begin{array}{llll}
[e,f]=h, & [h,e]=2e, & [h,f]=-2f,\\[1mm]
[h,x_i]=x_i, & [h,y_i]=-y_i,                    & [e,y_i]=x_i, & [f,x_i]=y_i. \\[1mm]
\end{array}$$

One can verify that the algebra $\mathfrak{g}_n$ can be regarded as the semidirect product of the simple Lie algebra $\mathfrak{sl}_2$ and the $2n$-dimensional abelian algebra $\mathfrak{a}_n= \langle x_1, \dots, x_n, y_1, \dots, y_n\rangle .$

First, we determine central extensions of the centerless $n$-th Schr\"{o}dinger algebra $\mathfrak{g}_n$ with trivial coefficients for $n\geq1$.

\begin{thm}{\label{thmg}}
$ \dim(H^2(\mathfrak{g}_n,\mathbb{C}))=\frac{n(n+1)}{2}$ for $n\geq1.$
\end{thm}

\begin{proof}

Let $n\geq1.$ Then $H^0(\mathfrak{sl}_2,\mathbb{C})\simeq\mathbb{C},$ $H^1(\mathfrak{sl}_2,\mathbb{C})=H^2(\mathfrak{sl}_2,\mathbb{C})=0$, and therefore the Hochschild-Serre factorization theorem (Theorem 13 in \cite{HchSe}) implies that
$$H^2(\mathfrak{g}_n,\mathbb{C})\simeq H^2(\mathfrak{a}_{n},\mathbb{C})^{\mathfrak{sl}_2}.$$

Since $\mathfrak{a}_{n},$ we easily get that $B^2(\mathfrak{a}_n,\mathbb{C})^{\mathfrak{sl}_2}=0$.
Let $\varphi\in Z^2(\mathfrak{a}_n,\mathbb{C})^{\mathfrak{sl}_2}.$
Considering $h.\varphi(x_i,x_j)=0$ and $h.\varphi(y_i,y_j)=0,$ we get that $\varphi(x_i, x_j) = 0$ and $\varphi(y_i, y_j)=0$ for any $i, j \in\{1,2,\dots,n\}.$
Moreover, the equality $$0=e.\varphi(y_i,y_j)= \varphi([e,y_i], y_j)+\varphi(y_i, [e,y_j])=\varphi(x_i, y_j)+\varphi(y_i, x_j)$$
implies that $\varphi(x_i, y_j)=\varphi(x_j, y_i)$ for $i\neq j.$
Summarizing the relations above, we deduce that $\dim(Z^2(\mathfrak{a}_n,\mathbb{C})^{\mathfrak{sl}_2})=\frac{n(n+1)}{2}$. Furthermore, by \eqref{iso} we have
$
H^2(\mathfrak{a}_n,\mathbb{C})^{\mathfrak{sl}_2}= Z^2(\mathfrak{a}_n,\mathbb{C})^{\mathfrak{sl}_2}/B^2(\mathfrak{a}_n,\mathbb{C})^{\mathfrak{sl}_2}.
$
Hence, $ \dim(H^2(\mathfrak{sch}_n,\mathbb{C}))=\frac{n(n+1)}{2}$.
\end{proof}

\begin{cor}
The set of anti-symmetric bilinear forms $\varphi_{ij} \ (i,j\in\{1,2,\dots,n\})$ satisfying \eqref{fij2} and \eqref{fij} is a basis for $H^2(\mathfrak{g}_n,\mathbb{C})$.
\end{cor}

Since $\dim(\mathfrak{g}_1)=5$, the direct computations show that $\dim(Z^2(\mathfrak{g}_1,\mathfrak{g}_1))=19$. On the other hand, by Remark 2.3 in \cite{LY}, we have $\dim({\rm Der}(\mathfrak{g}_1))=6$, and therefore $\dim(B^2(\mathfrak{g}_1,\mathfrak{g}_1))=5^2-\dim({\rm Der}(\mathfrak{g}_1)=19$. Thus, $H^2(\mathfrak{g}_1,\mathfrak{g}_1)=0$.

Now, we determine infinitesimal deformations of the Lie algebra $\mathfrak{g}_n$ for $n\geq3$.

\begin{thm}{\label{mainthm2}}
 $H^2(\mathfrak{g}_n,\mathfrak{g}_n)=0$ for $n\geq 3.$
\end{thm}

\begin{proof}
By the Hochschild-Serre factorization theorem (Theorem 13 in \cite{HchSe}) and Whitehead's first and second lemmas we have
$$H^2(\mathfrak{g}_n,\mathfrak{g}_n)\simeq
\mathbb{C}\otimes H^2(\mathfrak{a}_n,\mathfrak{g}_n)^{\mathfrak{sl}_2}\simeq H^2(\mathfrak{a}_n,\mathfrak{g}_n)^{\mathfrak{sl}_2}.$$

Let $\mathfrak{g}_n=\mathfrak{sl}_2\ltimes\mathfrak{a}_n$ be the Lie algebra defined as above. For a fixed cocycle $\varphi\in Z^{2}(\mathfrak{a}_n,\mathfrak{g}_n)^{\mathfrak{sl}_2}$, we need to define  $\varphi(x_i,x_j), \ \varphi(x_i,y_j), \ \varphi(y_i,y_j).$ Using the same $\mathfrak{sl}_2$-invariance arguments as in the proof of Lemma \ref{mainlem2}, we can easily get that $\varphi(x_i,y_j)=0, \ \varphi(x_i,x_j)=0, \ \varphi(y_i,y_j)=0.$ Hence, we derive $Z^2(\mathfrak{a}_n,\mathfrak{g}_n)^{\mathfrak{sl}_2}=0.$ Since $H^2(\mathfrak{a}_n,\mathfrak{g}_n)^{\mathfrak{sl}_2}=Z^2(\mathfrak{a}_n,\mathfrak{g}_n)^{\mathfrak{sl}_2}/B^2(\mathfrak{a}_n,\mathfrak{g}_n)^{\mathfrak{sl}_2}$, we obtain that $H^2(\mathfrak{a}_n,\mathfrak{g}_n)^{\mathfrak{sl}_2}=0.$
\end{proof}

\begin{cor}{\label{prop2}} The centerless $n$-th Schr\"{o}dinger algebra $\mathfrak{g}_n$ is rigid for  $ n \geq 3.$
\end{cor}



\begin{prop}{\label{prop2}}
 $ \dim(H^2(\mathfrak{g}_2,\mathfrak{g}_2))=1.$
\end{prop}

\begin{proof}
Let $\omega\in C^1(\mathfrak{a}_2,\mathfrak{g}_2)^{\mathfrak{sl}_2}.$ Then, the relations $h.\omega(x_i)=0$ and $h.\omega(y_i)=0$
imply that $\omega(x_i)\in\langle x_1,x_2\rangle$ and  $\omega(y_i)\in\langle y_1,y_2\rangle$ for $i=1,2.$
Moreover, from $e.\omega(y_i)=0,$ we get $[e,\omega(y_i)]=\omega(x_i)$ for $i=1,2.$
Thus, we get that any element $\omega\in C^1(\mathfrak{a}_2,\mathfrak{g}_2)^{\mathfrak{sl}_2}$ has the form
\begin{align*}
\omega(x_1) &= \alpha_1x_1 + \alpha_2x_2, & \omega(x_2) &= \alpha_3x_1 + \alpha_4x_2, \\
\omega(y_1) &= \alpha_1y_1 + \alpha_2y_2, & \omega(y_2) &= \alpha_3y_1 + \alpha_4y_2.
\end{align*}

Hence, $\dim(C^1(\mathfrak{a}_2,\mathfrak{g}_2)^{\mathfrak{sl}_2})=4.$ On the other hand, Since $\mathfrak{a}_2$ is abelian, any linear transformation on it is a derivation, so is $\omega.$ Hence, $B^2(\mathfrak{a}_2,\mathfrak{g}_n)^{\mathfrak{sl}_2}=0.$

Now, let $\varphi\in Z^2(\mathfrak{a}_2,\mathfrak{g}_2)^{\mathfrak{sl}_2}.$ Then, by the relation $h.\varphi(x_1,y_2)=0,$ we get that $\varphi(x_1,y_2)\in\langle h\rangle.$ In the same manner, by $h.\varphi(x_2,y_1)=0,$ we obtain $\varphi(x_2,y_1)\in\langle h\rangle.$ Moreover, from
$h.\varphi(x_1,x_2)=0, \ h.\varphi(y_1,y_2)=0,$
we can easily get that $\varphi(x_1,x_2)\in\langle e\rangle$ and  $\varphi(y_1,y_2)\in\langle f\rangle$, respectively.
 In addition, the relation $e.\varphi(x_1,y_1)=0$ implies that $[e,\varphi(x_1,y_1)]=0.$ Thus, $\varphi(x_1,y_1)=0.$ Similarly, by $e.\varphi(x_2,y_2)=0$, we get $\varphi(x_2,y_2)=0.$
Therefore, we can set
\begin{align*}
\varphi(x_1,x_2)&=\gamma e, &
\varphi(x_1,y_2)&=\alpha h,\\
\varphi(y_1,y_2)&=\mu f, & \varphi(x_2,y_1)&=\beta h.
\end{align*}

The relation $e.\varphi(x_1,y_2)=0$ implies that $\varphi(x_1,x_2)=[e,\varphi(x_1,y_2)].$ In the same manner, by $f.\varphi(x_1,y_2)=0$, we can get $\varphi(y_1,y_2)=[f,\varphi(x_1,y_2)].$ These imply that $\gamma=-2\alpha$ and $\mu=2\alpha.$ Furthermore, the relations $e.\varphi(x_2,y_1)=0,$ $f.\varphi(x_2,y_1)=0$ leads to $\varphi(x_1,x_2)=-[e,\varphi(x_2,y_1)]$ and $\varphi(y_1,y_2)=-[f,\varphi(x_2,y_1)],$ respectively. Thus, $\beta=-\alpha.$

Summarizing these relations, we conclude that $\varphi$ is of the following form:
\begin{align*}
\varphi(x_1,x_2)&=-2\alpha e, &
\varphi(x_1,y_2)&=\alpha h,\\
\varphi(y_1,y_2)&=2\alpha f, &
\varphi(x_2,y_1)&=-\alpha h.
\end{align*}

Therefore, $\dim(Z^2(\mathfrak{a}_2,\mathfrak{g}_2)^{\mathfrak{sl}_2})=1.$ Moreover, by \eqref{iso} we have $$H^2(\mathfrak{a}_2,\mathfrak{sch}_2)^{\mathfrak{sl}_2}\simeq Z^2(\mathfrak{a}_2,\mathfrak{sch}_2)^{\mathfrak{sl}_2}/B^2(\mathfrak{a}_2,\mathfrak{sch}_2)^{\mathfrak{sl}_2}$$
which follows that $\dim(H^2(\mathfrak{a}_2,\mathfrak{g}_2)^{\mathfrak{sl}_2})=1.$ Hence, $\dim(H^2(\mathfrak{g}_2,\mathfrak{g}_2))=1.$
\end{proof}

\begin{cor}
The alternating bilinear map $\psi$ defined as below, forms a basis on $H^2(\mathfrak{g}_2,\mathfrak{g}_2).$
\begin{align*}
\psi(x_1,x_2)&=2e,  & \psi(y_1,y_2)&=-2f,\\
\psi(x_1,y_2)&=-h,  & \psi(x_2,y_1)&=h.
\end{align*}

\end{cor}


\begin{thebibliography}{99}












\bibitem{Ru} Campoamor-Stursberg R., Cohomological rigidity of the Schr\"odinger algebra $S(N)$ and its central extension $\hat{S}(N)$, J. Lie Theory, 27 (2017), 315-328.

\bibitem{CEil} Cartan H., Eilenberg S., Homological Algebra, Princeton University Press, Princeton, NJ, 1956.





\bibitem{ChEg} Chevalley C., Eilenberg S., Cohomology theory of Lie groups and Lie algebras, Trans. Amer. Math. Soc., 63 (1948), 85-124.


\bibitem{Makhlouf}  Ehret Q., Makhlouf A., On classification and deformations of Lie-Rinehart superalgebras,  Comm. in Math., 30 (2022), 67-92.



\bibitem{Gest} Gerstenhaber M., On the Deformation of Rings and Algebras. Ann. Math., 79 (1964), 1, 59-103.

\bibitem{HchSe} Hochschild G., Serre J.-P., Cohomology of Lie algebras, Ann. Math., (2) 57 (1953),  591-603.




\bibitem{Ksz} Koszul J.-L., Homologie et cohomologie des alg\`{e}bres de Lie, Bull. Soc. Math. France, 78 (1950), 65-127.


\bibitem{LY} Lei B., Yang H., The derivation algebra and automorphism group of the $n$-th Schr\"odinger algebra,  Comm. in Algebra, 52 (2024), 283-294.


\bibitem{NR} Nijenhuis A., Richardson R. W. Jr., Deformations of Lie algebra structures, J. Math. Mech., 17 (1967), 89-105.







\bibitem{WuZh} Wu Y.Zh., Yue X.Q., Zhu L.Sh., Cohomology of the Schr\"{o}dinger Algebra $\mathcal{S}(1),$ Acta Mathematica Sinica, 30 (2014), 2054-2062.


\bibitem{WuT} Wu Q., Tang X., Derivations and biderivations of the Schr\"{o}dinger algebra in $(n+1)$-dimensional space-time, Linear Multilin. Algebra, 71 (2023), 1073-1097.





















\end{thebibliography}
\end{document}